\documentclass[12pt,reqno]{amsart}

\usepackage[mathlines]{lineno}

\usepackage{times}
\usepackage{amsfonts}
\usepackage{amssymb}
\usepackage{latexsym}
\usepackage{xcolor}

\newtheorem{theorem}{Theorem}
\newtheorem{lemma}[theorem]{Lemma}

\newtheorem{definition}[theorem]{Definition}

\newtheorem{corollary}[theorem]{Corollary}

\begin{document}
\begin{center}
On the symbol calculus for multidimensional Hausdorff operators
\end{center}

\bigskip

\begin{center}
E. Liflyand
\end{center}
\begin{center}
Department of Mathematics, Bar-Ilan University, Ramat-Gan 52900, Israel; \quad e-mail: liflyand@gmail.com
\end{center}

\begin{center}
and
\end{center}

\begin{center}
A. Mirotin
\end{center}
\begin{center}
Department of Mathematics and Programming Technologies, Francisk
Skorina Gomel State University, Gomel, Belarus, and
Regional Mathematical Center, Southern Federal University, Rostov-on-Don, Russia; \quad e-mail: amirotin@yandex.by
\end{center}

\medskip

{\bf 2020 MSC:} Primary 47B38; Secondary 42B10

{\bf Keywords:} Hausdorff operator; commuting family; commutative algebra; symbol; matrix symbol; Fourier transform; convolution; positive definiteness; holomorphic function, fractional power

\bigskip

{\bf Abstract.} The aim of this work is  to derive a symbol calculus on $L^2(\mathbb{R}^n)$ for multidimensional  Hausdorff operators.
Two aspects of this activity result in two almost independent parts. While throughout the perturbation matrices are supposed to be self-adjoint
and form a commuting family, in the second part they are additionally assumed to be positive definite. What relates these two parts is the powerful
method of diagonalization  of a normal Hausdorff operator elaborated earlier by the second named author.

\bigskip

\section{Introduction}

Not diving too deep for the history of the modern Hausdorff operators on the Euclidean spaces, which usually starts with the reference to Hardy's book \cite{Ha},
where accurate reference to independent pioneer works by Rogosinski can be found (in fact, the same results were independently obtained by Garabedian), we instead refer the reader to the survey papers
\cite{emj} and \cite{CFW}. Even more recent and relevant is \cite{KL}, where an attempt is undertaken to figure out what the notion of Hausdorff operator in Euclidean spaces means. We will use a version of the definition given there.

\begin{definition} Let  $K$  be a locally Lebesgue integrable function on $\mathbb{R}^n$  and let $(A(u))$, with $u\in \mathbb{R}^n$, be a measurable family of real
$(n\times n)$-matrices in  the support of  $K$ almost everywhere defined and satisfying $\mathrm{det}A(u) \ne 0$.  We define the
\textit{Hausdorff  operator} $\mathcal{H}_{K, A}$ with kernel $K$  by

\begin{equation}\label{hodef} \mathcal{H}_{K, A}f(x) =\int_{\mathbb{R}^n} K(u)f(A(u)x)du,\ x\in\mathbb{R}^n\ \mbox{\rm is\ a\ column\ vector}.\end{equation}
\end{definition}

In the sequel, we shall assume  that  $|\det A(u)|^{-\frac12}K(u)\in L^1(\mathbb{ R}^n)$. This guarantees the boundedness of $\mathcal{H}_{K, A}$
in  $ L^2(\mathbb{ R}^n)$ and induces the introduction of the set

$${\mathcal A}_A:=\Big\{\mathcal{H}_{K, A}: |\det A(u)|^{-\frac12}K(u)\in L^1(\mathbb{ R}^n)\Big\}$$
of bounded operators in $ L^2(\mathbb{ R}^n)$.

One more assumption throughout the paper will be that the matrices $A(u)$ {\it are self-adjoint and  form a commuting family}. This implies that
there are an orthogonal $n\times n$-matrix $C$ and a family of diagonal non-singular real matrices $A'(u)={\rm diag}(a_1(u),...,a_n(u))$
such that $A'(u)=C^{-1}A(u)C$ for all $u\in \mathbb{R}^n$, where $A(u)$ is defined. By this, $a(u):=(a_1(u),...,a_{n}(u))$ is the family of all the
eigenvalues (with their multiplicities) of the matrix $A(u)$.

It is known that in this case the Hausdorff operator $\mathcal{H}_{K, A}$  in $L^2(\mathbb{R}^{n})$  is normal \cite{Forum}.

The aim of this work is  to extend the symbol calculus for one-dimensional Hausdorff operators on $L^2(\mathbb{R})$ elaborated in \cite{LM2}
to the multivariate case. It is not a plain business, and in order to use the approach in \cite{LM2} various constraints are posed on the
matrices $A(u)$. This leaves a room for further research in the cases where both the results and the methods of proof are questionable.

It is worth noting one more peculiarity of this work. In fact, a complete definition of Hausdorff operators in \cite{KL} suggests that
(\ref{hodef}) allows $u\in\mathbb{R}^m$, with $m$ not necessarily equal to $n$. To this end, we mention \cite{LiMi1}, where in certain
cases it is necessary that $m>n$. On the other hand, in some problems the case $m<n$ may also be meaningful; see, e.g., \cite{An}.

There are three main results in this work, Theorems \ref{thm1},  \ref{thm2}, and \ref{thm3}. We formulate, prove and discuss them and their
consequences in the two following sections.   All considerations are based on the diagonalization  of a normal Hausdorff operator
obtained in \cite{Forum} and \cite{faa}.

\section{The algebra ${\mathcal A}_A$}

Prior to the formulation and proof of the result in this section, we need more preliminaries. We split $\mathbb{R}^n$ into $2^n$
hyper-octants $\mathbb{R}^n_i$, fixing also an enumeration of this family. For every pair of the indices $(i,j)$, $i,j=1,\dots, 2^n$, there is a unique
$\varepsilon(i,j)\in\{-1,1\}^n$ such that

$$\varepsilon(i,j)\,\mathbb{R}^n_i:=\{(\varepsilon(i,j)_1\,u_1,...,\varepsilon(i,j)_n u_n): u\in\mathbb{R}^n_i\}=\mathbb{R}^n_j.$$
It is worth noting that $\varepsilon(i,j)\,\mathbb{R}^n_j=\mathbb{R}^n_i $ and $\varepsilon(i,j)\,\mathbb{R}^n_l\cap\mathbb{R}^n_i=\emptyset$ whenever $l\not=j$. We put

$$\Omega_{ij}:=\Big\{ u\in \mathbb{R}^n: \left(\mathrm{sgn}\,(a_1(u)),...,\mathrm{sgn}\,(a_n(u))\right)=\varepsilon (i,j)\Big\}.$$

For $\mathcal{H}_{K, A}\in {\mathcal A}_A$, assuming that

$$|a(u)|^{-\frac12 -\imath s}:=\prod\limits_{k=1}^n |a_k(u)|^{-\frac12-\imath s_k},\ s=(s_1,...,s_n)\in \mathbb{R}^n,$$
with
$$|a_k(u)|^{-\frac12-\imath s_k}:=e^{-(\frac12+\imath s_k)\log|a_k(u)|},$$
we define

$$\varphi_{ij}(s):=\int_{\Omega_{ij}} K(u)\,|a(u)|^{-\frac12-\imath s}\,du.$$
Obviously, each $\varphi_{ij}\in C_b(\mathbb{R}^n)$ and $\varphi_{ij}=\varphi_{ji}$.

Following \cite{Forum} (see also \cite{faa}) we are now in a position to define the {\it matrix symbol} of the Hausdorff operator $\mathcal{H}_{K, A}\in {\mathcal A}_A$ by

\begin{equation}\label{ms} \Phi=\Big(\varphi_{ij}\Big)_{i,j=1}^{2^n}.\end{equation}
By this, $\Phi$ is a symmetric element of the matrix algebra $\mathrm{Mat}_{2^n}\Big(C_b(\mathbb{R}^n)\Big)$.

We also need a property of the map defined above.

\begin{lemma}\label{lm1} The map $\mathrm{Smb}: \mathcal{A}_A\to \mathrm{Mat}_{2^n}(C_0(\mathbb{R}^n)),$
where $C_0(\mathbb{R}^n)$ stands for the algebra of continuous functions on $\mathbb{R}^n$ vanishing at infinity,
is an isometry, if we endow the algebra $ \mathrm{Mat}_{2^n}(C_0(\mathbb{R}^n))$ with the norm $\|\Phi\|=\sup_{s\in \mathbb{R}^n}\|\Phi(s)\|_{op}$.
\end{lemma}

Here    $\|\cdot\|_{op}$ stands for the operator norm of a matrix as the norm of the operator of multiplication by this matrix.

\begin{proof} Let $ M_{\Phi}$ denote the operator of multiplication by the matrix function  $\Phi\in \mathrm{Mat}_{2^n}(C_0(\mathbb{R}^n))$ in the space of
vector valued functions $L^2(\mathbb{R}^n, \mathbb{C}^{2^n})$. It is known (see \cite{faa} and  \cite{Forum}) that the map $\mathcal{H}_{K,A}\mapsto M_{\Phi}$
is an isometry (with respect to operator norms) if $\Phi=\mathrm{Smb}(\mathcal{H}_{K,A})$. On the other hand,  $\| M_{\Phi}\|=\|\Phi\|$ by \cite[Corollary 3]{Forum}.
\end{proof}

In this section, we will  assume  that  for each pair of indices $i,j$, the system of equations

\begin{equation}\label{syeq}
|a_k(u)|=e^{t_k},  k=1,\dots, n,
\end{equation}
has the unique solution $u=(b_1(t),\dots,b_n(t))\in \Omega_{ij}$, $t=(t_1,...,t_n)$, for almost  every $t_k\in \mathbb{R}$.
Hence, we have a measurable map $ \mathbb{R}^n\to \Omega_{ij}$, $t\mapsto b(t)$, which is almost bijective.

Finally, we are ready to formulate and prove our first main result.

\begin{theorem}\label{thm1}   Under the above assumptions the set $\mathcal{A}_A$ is a non-closed  commutative subalgebra
without unit of the Banach algebra $\mathcal{L}(L^2(\mathbb{R}^n))$ of bounded operators on $L^2(\mathbb{R}^n)$.
\end{theorem}

\begin{proof}  Straightforward calculations yield the commutativity of $\mathcal{A}_A$.

Putting  $|a_k(u)|=e^{t_k}$, we get, since
$$ |\det A(u)|=\prod_{k=1}^{n} |a_k(u)|=e^{\sum\limits_{k=1}^{n} t_k},$$
 that ($s\in\mathbb{R}^n$)
\begin{eqnarray*}
   \varphi_{ij}(s)&=\int_{\mathbb{R}^n}K(b_1(t),\dots, b_n(t))e^{-\frac{1}{2}\sum\limits_{k=1}^n t_k}|J(t)| e^{- \imath s\cdot t}\,dt=\widehat{K_{ij}}(s),
\end{eqnarray*}
where the ``hat'' stands for the Fourier transform, $t=(t_1,...,t_n)$,
$$
{K_{ij}}(t):=K(b_1(t),\dots, b_n(t))e^{-\frac{1}{2}\sum\limits_{l=1}^n t_l}|J(t)|,
$$
and $J(t):=\frac{\partial(b_1,\dots,b_n)}{\partial(t_1,\dots,t_n)}$ is the Jacobian.

Since the map

\begin{equation}\label{map}
\mathrm{Smb}:\mathcal{H}_{K,A}\mapsto \Phi, \quad \mathcal{A}_A\to \mathrm{Mat}_{2^n}(C_0(\mathbb{R}))
\end{equation}
is an isometry (and therefore, injective) and multiplicative (see \cite{Forum} for details), to prove that $\mathcal{A}_A$ is an algebra,
it suffices to show that the product of two symbols is also a symbol. More precisely, it suffices to show that
if  $\mathrm{Smb}(\mathcal{H}_{K,A})=\Phi$ and $\mathrm{Smb}(\mathcal{H}_{L,A})=\Psi$,
then $\Phi\Psi=\mathrm{Smb}(\mathcal{H}_{Q,A})$ for some $\mathcal{H}_{Q,A}\in \mathcal{A}_A$.

Denoting $\psi_{ij}$ and $L_{ij}$ for $\Psi$ similarly to $\varphi_{ij}$ and $K_{ij}$ for $\Phi$ and replacing the notation $\quad\widehat{}\quad$ for the Fourier transform by $\mathcal{F}$, we have

\begin{align*} \Phi\Psi&=(\varphi_{ij})(\psi_{ij})=\biggl(\,\sum\limits_{k=1}^{2^n} \varphi_{ik}\psi_{kj}\,\biggr)_{i,j=1}^{2^n}\\
&=\biggl(\,\mathcal{F}\Big(\,\sum\limits_{k=1}^{2^n} K_{ik}\ast L_{kj}\,\Big)\,\biggr)_{i,j=1}^{2^n},\end{align*}
where $\psi_{kj}=\widehat{L_{kj}}$, $L_{kj}\in L^1(\mathbb{R}^n)$, and $\ast$ denotes the convolution in $L^1(\mathbb{R}^n)$.

Defining the functions $Q_{ij}$ on $\mathbb{R}^n$ by

$$ Q_{ij}:=\sum\limits_{k=1}^{2^n} K_{ik}\ast L_{kj},$$
we arrive at

$$ \Phi\Psi=\Big(\widehat{Q_{ij}}\Big)_{i,j=1}^{2^n}.$$
Let now $Q$ be a function on $\mathbb{R}^n$ satisfying

$$Q_{ij}(t)=Q(b_1(t),\dots, b_n(t))e^{-\frac{1}{2}\sum\limits_{l=1}^n t_l}|J(t)|.$$
By this, $\Phi\Psi=\mathrm{Smb}(\mathcal{H}_{Q,A})$ in accordance with the above arguments.  
Since $Q_{ij}\in L^1(\mathbb{R}^n)$, we have $|\det A(u)|^{-\frac12}Q(u)\in L^1(\mathbb{ R}^n)$. Hence, $\mathcal{H}_{Q,A}\in \mathcal{A}_A$.

\medskip

Let us now proceed to the non-closedness. To this end, we choose
a sequence of kernels $K_\nu$ satisfying $|\det A(u)|^{-\frac12}K_\nu(u)\in L^1(\mathbb{ R}^n)$ and such that the sequence of Fourier transforms $\widehat{K_{\nu,11}}$
converges to a function in $C_0(\mathbb{R}^n)\setminus W_0(\mathbb{R}^n)$ uniformly on $\mathbb{R}^n$. Here $ W_0(\mathbb{R}^n)$ denotes the Wiener
algebra of Fourier transforms of functions in $L^1(\mathbb{R}^n)$; for a comprehensive survey, see \cite{LST}. Assume that the sequence of operators $\mathcal{H}_{K_\nu,A}$
converges to an operator  $\mathcal{H}_{L,A}$ in $ \mathcal{A}_A$ in the operator norm. Then by Lemma \ref{lm1}, the  sequence  of symbols
$\mathrm{Smb}(\mathcal{H}_{K_\nu,A})$ converges in the norm $\|\cdot\|_{op}$ to $\mathrm{Smb}(\mathcal{H}_{L,A})$ uniformly on $\mathbb{R}^n$.
Since the convergence in the norm in a finite-dimensional space implies the coordinate-wise convergence, this implies that $\widehat{K_{\nu,11}}$ converges (at least point-wise) to  $\widehat{L_{11}}\in W_0(\mathbb{R}^n)$  on $\mathbb{R}^n$, and we arrive at a contradiction.

Finally,  let $\mathcal{H}_{K,A}=I$, the identity operator for some $\mathcal{H}_{K,A}\in \mathcal{A}_A$. Then $\mathrm{Smb}(\mathcal{H}_{K,A})=E_{2^n}$
(the unit  matrix of order $2^n$). Therefore, $\widehat{K_{ii}}(s)=1$, a contradiction.    This completes the proof.
   \end{proof}

\begin{corollary} The normed algebra $\mathcal{A}_A$ is not Banach.   \end{corollary}

The above treatment leads to the open

\noindent {\bf Problem.} What is the closure of   $\mathcal{A}_A$ in the uniform operator topology?

\section{Functions of a Hausdorff operator}

We begin with an assumption used throughout the preceding section. However, it is simplified here because
of the positive definiteness condition posed on the matrices.

Let   $a(u):=(a_1(u),\dots,a_n(u))$ be the family of eigenvalues  (with their multiplicities) of a  {\it positive definite}  matrix $A(u)$.
We will assume in this section that the system of equations
$$
a_k(u)=e^{t_k},  k=1,\dots, n
$$
has the unique solution $u=(b_1(t),\dots,b_n(t))$, $t=(t_1,...,t_n)$, for almost  every $t_k\in \mathbb{R}$.
Again, we have a measurable map $ b:\mathbb{R}^n\to  \mathbb{R}^n$,  which is almost bijective.

\medskip

\subsection{Holomorphic functions of a Hausdorff operator}

\quad

The following theorem is a variant of functional calculus for Hausdorff operators.

\begin{theorem}\label{thm2} Let $K$ and $A$ satisfy the conditions listed above and let each matrix $A(u)$ be positive definite.
If a function $F$ is holomorphic in the neighborhood $E$ of the spectrum  $\sigma(\mathcal{H}_{K,A})$ and $F(0)=0$, then $F(\mathcal{H}_{K,A})$
is also a  Hausdorff operator of the form $\mathcal{H}_{K_F,A}$ bounded in $L^2(\mathbb{R}^n)$.
\end{theorem}

\begin{proof} Let
 $$
   \varphi_{K,A}(s):=\int_{\mathbb{ R}^n}K(u)a(u)^{-\frac12- \imath s}du, \quad s=(s_1,...,s_n)\in \mathbb{R}^n,
   $$
stand for  the symbol of the Hausdorff operator $\mathcal{H}_{K, A}$ in $L^2(\mathbb{R}^n)$ \cite{faa, Forum}.

Recall that, by our definitions,  $a(u)^{-\frac12-\imath s}:=$ $\prod\limits_{k=1}^n a_k(u)^{-\frac12-\imath s_k}$,
 where $a_k(u)^{-\frac12-\imath s_k}:=$ $\exp((-\frac12-\imath s_k)\log a_k(u))$.

Note that
\begin{eqnarray*}
   \varphi_{K,A}(s)&=&\int_{\mathbb{ R}^n}K(u)\prod_{k=1}^na_k(u)^{-\frac12} e^{- \imath s_k\log a_k(u)}du\\
   &=&\int_{\mathbb{ R}^n}K(u)(\det A(u))^{-\frac12} e^{- \imath s\cdot \log a(u)}du,
\end{eqnarray*}
where
$$s\cdot \log a(u):=\sum_{k=1}^ns_k \log a_k(u).
$$
 Putting  $a_k(u)=e^{t_k}$,  with $t=(t_1,...,t_n)\in \mathbb{R}^n$,  we get,
since $\det A(u)=\prod\limits_{k=1}^na_k(u)=e^{\sum\limits_{k=1}^n t_k}$, that
\begin{eqnarray*}
   \varphi_{K,A}(s)&=\int_{\mathbb{ R}^n}K(b_1(t),\dots, b_n(t))e^{-\frac{1}{2}\sum_{k=1}^nt_k}|J(t)| e^{- \imath s\cdot t}dt=\widehat L(s),
\end{eqnarray*}
where
$$
L(t):=K(b_1(t),\dots, b_n(t))e^{-\frac{1}{2}\sum\limits_{k=1}^n t_k}|J(t)|,
$$
and $J(t):=\frac{\partial(b_1,\dots,b_n)}{\partial(t_1,\dots,t_n)}$ is the Jacobian.

According to \cite[Theorem 1]{faa}, \cite{Forum}, the  Hausdorff operator $\mathcal{H}_{K,A}$   in $L^2(\mathbb{R}^n)$   is
unitary equivalent to the operator $M_{\varphi_{K,A}}$  of coordinate-wise multiplication by  $\varphi_{K,A}$ in the space
$L^2(\mathbb{R}^n,\mathbb{C}^{2^n})$ of $\mathbb{C}^{2^n}$-valued functions. More precisely,  $\mathcal{H}_{K, A}=\mathcal{V}^{-1}
M_{\varphi_{K,A}} \mathcal{V}$, where  $\mathcal{V}$ is a unitary operator between $L^2(\mathbb{R}^n)$ and $L^2(\mathbb{R}^n,\mathbb{C}^{2^n})$
independent of $K$.  Moreover, the   spectrum $\sigma(\mathcal{H}_{K, A})$ is equal to the closure of the range of  the   symbol $\varphi_{K,A}$.

Then, as in \cite[Theorem 2]{LM2}, we have
\begin{align}\label{F(H)}
F(\mathcal{H}_{K,A})&=\frac{1}{2\pi \imath}\int_{\Gamma} F(\lambda)(\lambda-\mathcal{H}_{K,A})^{-1}d\lambda \nonumber\\
&=\frac{1}{2\pi \imath}\int_{\Gamma} F(\lambda)(\lambda-\mathcal{V}^{-1}M_{\varphi_{K,A}}\mathcal{V})^{-1}d\lambda \nonumber\\
&=\mathcal{V}^{-1}\left(\frac{1}{2\pi \imath}\int_{\Gamma} F(\lambda)(\lambda - M_{\varphi_{K,A}})^{-1}d\lambda\right)\mathcal{V} \nonumber\\
&=\mathcal{V}^{-1}F(M_{\varphi_{K,A}})\mathcal{V}=\mathcal{V}^{-1}M_{F(\varphi_{K,A})}\mathcal{V},
\end{align}
where  $\Gamma$ is the boundary of any open neighborhood  of the  set  $\sigma(\mathcal{H}_{K,A})$ such that $E$ contains its closure.

Here the identity
\[
(\lambda-\mathcal{V}^{-1}M_{\varphi_{K,A}}\mathcal{V})^{-1}=(\mathcal{V}^{-1}\lambda \mathcal{V}-\mathcal{V}^{-1}M_{\varphi_{K,A}}\mathcal{V})^{-1}=
\mathcal{V}^{-1}(\lambda-M_{\varphi_{K,A}})^{-1}\mathcal{V}
\]
and properties of the Bochner integral are used.

To finish the proof, it remains to show that $F(\varphi_{K,A})$ is the symbol of some Hausdorff operator $\mathcal{H}_{K_F,A}$ (this operator is bounded
in $L^2(\mathbb{R}^n)$ by the holomorphic functional calculus \cite{DunSw}).

To this end, observe that $F(\varphi_{K,A})\equiv F(\widehat L)=\widehat {Q_F}$ for some $Q_F\in L^1(\mathbb{R}^n)$ (see, e.~g., \cite[Theorem 6.2.4]{Rudin}; this theorem is applicable, since $F$ is holomorphic on the open set $E$, which contains the closure of the range of $\varphi_{K,A}$). Thus, if we denote $\log a(u):=(\log a_1(u),\dots,\log a_n(u))$ and put
\begin{eqnarray}\label{KF}
K_F(u):=(\det A(u))^{\frac12}\frac{Q_F(\log a(u))}{|J(\log a(u))|},
\end{eqnarray}
then
$$
K_F(b_1(t),\dots,b_n(t))=e^{\frac12\sum\limits_{k=1}^{n}t_k}\frac{Q_F(t)}{|J(t)|},
$$
and
\begin{align}\label{phiK}
   \varphi_{K_F,A}(s)&=\int_{\mathbb{ R}^n}K_F(u)a(u)^{-\frac12- \imath s}du\nonumber\\
&=\int_{\mathbb{ R}^n}K_F(b_1(t),\dots, b_n(t))e^{(-\frac12)\sum_{k=1}^nt_k}|J(t)| e^{- \imath s\cdot t}dt \nonumber\\
&=\widehat Q_F(s)=F(\varphi_{K,A})(s).
 \end{align}

Therefore, by \eqref{F(H)}, we have $F(\mathcal{H}_{K,A})=\mathcal{H}_{K_F,A}$, as desired.        \end{proof}

\bigskip

{\bf Example 1}. Let $F(z)=z^l, l\in \mathbb{N}, l\ge 2$.  Then $\mathcal{H}_{K,A}^l$ equals to some Hausdorff operator $\mathcal{H}_{K_l,A}$
with a scalar symbol $\widehat Q_l= \varphi_{K,A}^l$, by \eqref{phiK}.

We now consider the  averaging operator of Boyd \cite{Boyd1}, \cite{Boyd}

\[
P_\alpha f(x)=x^{\alpha-1}\int_0^xt^{-\alpha}f(t)dt=\int_0^1u^{-\alpha}f(ux)\,du,
\]
where $P_0=\mathcal{C}$ is the continuous Ces\'{a}ro operator. For $\alpha<\frac12$, this is a bounded Hausdorff operator in $L^2(\mathbb{R})$ with
the kernel $K(u)=\chi_{(0,1)}(u)u^{-\alpha}$, where $\chi$ stands for the indicator function of the set indicated as a subscript, and $A(u)=a(u)=u$. Its symbol is
\[
\varphi_{P_\alpha}(s)=\int_{\mathbb{R}}\chi_{(0,1)}(u)u^{-\alpha-\frac{1}{2}-\imath s}du=\frac{1}{(\frac{1}{2}-\alpha)- \imath s}.
\]
As mentioned above,
\[
\widehat Q_l(s)= \varphi_{P_\alpha}^l(s)=\frac{1}{((\frac{1}{2}-\alpha)- \imath s)^l}.
\]
Formula (3) in \cite[Ch. III, \S 3.2]{BE} yields 
\[
Q_l(t)=\frac{(-1)^{l-1}}{(l-1)!}t^{l-1}e^{(\frac{1}{2}-\alpha)t}\chi_{(-\infty,0)}(t).
\]

Since in our case $J(t)=e^t$, formula \eqref{KF} implies
\begin{eqnarray*}
K_l(u)&=&u^{\frac{1}{2}}\frac{Q_l(\log u)}{J(\log u)}\\
&=&\frac{u^{\frac{1}{2}}}{u}\frac{(-1)^l}{(l-1)!}(\log u)^{l-1}e^{(\frac{1}{2}-\alpha)\log u}\chi_{(-\infty,0)}(\log u)\\
&=&\frac{1}{(l-1)!}\left(\log\frac{1}{u}\right)^{l-1}u^{-\alpha}\chi_{(0,1)}(u).
\end{eqnarray*}
Thus,
\[
P_\alpha^l f(x)=\frac{1}{(l-1)!}\int_0^1u^{-\alpha}\left(\log\frac{1}{u}\right)^{l-1}f(ux)du.
\]
For $x> 0$, this is a formula of Boyd \cite[Lemma 2]{Boyd}\footnote{Boyd's formula is important for the study of the dynamics of $\mathcal{C}$,
see, e.~g. \cite{GFS} and the bibliography therein.}. As follows from our considerations, it is valid for all $f\in L^2(\mathbb{R})$ if $\alpha<\frac12$.
Since for $\alpha<1-\frac1p$, the operator $P_\alpha$ is bounded in $L^p(\mathbb{R})$ with  $p\in(1,\infty)$, by the Minkowski inequality,
for such $\alpha$ Boyd's formula is valid for all $f\in L^p(\mathbb{R})$, $p\in(1,\infty)$, as well.

\medskip

\subsection{Fractional powers of Hausdorff operator}

\quad

Let $\rm{Re}\,\alpha>0 $. Since the function $z^\alpha$ is not holomorphic in any neighborhood of zero, the approach of the previous subsection
is not applicable to the case where $0\in\sigma(\mathcal{H}_{K,A})$ and needs a special treatment.

For fractional power of a non-negative bounded operator $B$ in the Hilbert space $\frak{H}$, we will make
use of the following formula \cite[Ch. 3, Proposition 3.1.1; Ch. 5, Definition 5.1.1]{MS}). For every positive integer $m>\rm{Re}\,\alpha$,
$$
B^\alpha f=\frac{\Gamma(m)}{\Gamma(\alpha)\Gamma(m-\alpha)} \int_0^\infty t^{\alpha-1}(B(t+B)^{-1})^m f dt,\quad f\in \frak{H}.
$$

\begin{theorem}\label{thm3} Let $\rm{Re}\,\alpha>0 $, $A(u)$ satisfy the above conditions and let, in addition, each matrix $A(u)$ be positive definite  for a.~e. $u$.
Let the scalar symbol $\varphi_{K,A}\ge 0$ and the fractional power $\varphi_{K,A}^\alpha$ be the Fourier transform of a function $Q_\alpha\in L^1(\mathbb{R}^n)$.  Then
the fractional power $\mathcal{H}_{K,A}^\alpha$ is also a  Hausdorff operator of the form $\mathcal{H}_{K_{\alpha},A}$ bounded in $L^2(\mathbb{R}^n)$.
\end{theorem}

\begin{proof} We first note that the  Hausdorff operator $\mathcal{H}_{K,A}$ is normal and its spectrum $\sigma(\mathcal{H}_{K,A})$ equals to the closure of the range of
 $\varphi_{K,A}$ \cite{faa, Forum}. Since $\varphi_{K,A}\ge 0$, we conclude that the operator $\mathcal{H}_{K,A}$ is non-negative in $L^2(\mathbb{R}^n)$.
 Thus, for $f\in L^2(\mathbb{R}^n)$ and every positive integer $m>\rm{Re}\,\alpha$,   we have
\begin{eqnarray*}
 \mathcal{H}_{K,A}^\alpha f=\frac{\Gamma(m)}{\Gamma(\alpha)\Gamma(m-\alpha)} \int_0^\infty t^{\alpha-1}(\mathcal{H}_{K,A}(t+\mathcal{H}_{K,A})^{-1})^m f dt.
\end{eqnarray*}
As in the proof of the previous theorem, $\mathcal{H}_{K, A}=\mathcal{V}^{-1}
M_{\varphi_{K,A}} \mathcal{V}$, where  $\mathcal{V}$ is a unitary operator taking $L^2(\mathbb{R}^n)$ onto $L^2(\mathbb{R}^n,\mathbb{C}^{2^n})$
independent of $K$. Then
 \begin{align*}
(\mathcal{H}_{K,A}(t+\mathcal{H}_{K,A})^{-1})^m&=(\mathcal{V}^{-1}M_{\varphi_{K,A}} \mathcal{V}(t+\mathcal{V}^{-1}
M_{\varphi_{K,A}} \mathcal{V})^{-1})^m\\
&=(\mathcal{V}^{-1} M_{\varphi_{K,A}} (t+M_{\varphi_{K,A}})^{-1}\mathcal{V})^m\\
&=\mathcal{V}^{-1}(M_{\varphi_{K,A}}(t+M_{\varphi_{K,A}})^{-1})^m\mathcal{V}.
\end{align*}

Here the identity
 \begin{align*}
\mathcal{V}^{-1} M_{\varphi_{K,A}}\mathcal{V} (t+\mathcal{V}^{-1}M_{\varphi_{K,A}}\mathcal{V})^{-1}&=\mathcal{V}^{-1} M_{\varphi_{K,A}}\mathcal{V}(\mathcal{V}^{-1} (t+M_{\varphi_{K,A}})\mathcal{V}) ^{-1}\\
&=\mathcal{V}^{-1} M_{\varphi_{K,A}}(t+M_{\varphi_{K,A}}) ^{-1}\mathcal{V}
\end{align*}
is used. Therefore, by \cite[Example 3.3.1]{MS},
 \begin{align}\label{H_to_alpha}
 &\mathcal{H}_{K,A}^\alpha f=\frac{\Gamma(m)}{\Gamma(\alpha)\Gamma(m-\alpha)} \int_0^\infty t^{\alpha-1}
 \mathcal{V}^{-1}(M_{\varphi_{K,A}}(t+M_{\varphi_{K,A}})^{-1})^m\mathcal{V} f\, dt \nonumber\\
&=\mathcal{V}^{-1}\left(\frac{\Gamma(m)}{\Gamma(\alpha)\Gamma(m-\alpha)} \int_0^\infty t^{\alpha-1}
(M_{\varphi_{K,A}}(t+M_{\varphi_{K,A}})^{-1})^m  \mathcal{V} f\, dt \right) \nonumber\\
&=\mathcal{V}^{-1} M_{\varphi_{K,A}}^\alpha\mathcal{V}f=\mathcal{V}^{-1} M_{\varphi_{K,A}^\alpha}\mathcal{V}f.
\end{align}
Since $\varphi_{K,A}^\alpha=\widehat {Q_\alpha}$ for some $Q_\alpha\in L^1(\mathbb{R}^n)$, we can proceed as in the proof of the previous theorem. Indeed, let
\begin{eqnarray}\label{K_alpha}
K_{\alpha}(u):=(\det A(u))^{\frac12}\frac{Q_\alpha(\log a(u))}{|J(\log a(u))|}.
\end{eqnarray}
Then
$$
K_\alpha(b_1(t),\dots,b_n(t))=e^{\frac12\sum\limits_{k=1}^{n}t_k}\,\frac{Q_\alpha(t)}{|J(t)|},
$$
and
\begin{align*}
   \varphi_{K_\alpha,A}(s)&=\int_{\mathbb{ R}^n}K_\alpha(u)a(u)^{-\frac12- \imath s}du\\
&=\int_{\mathbb{ R}^n}K_\alpha(b_1(t),\dots, b_n(t))e^{-\frac12\sum\limits_{k=1}^nt_k}|J(t)| e^{- \imath s\cdot t}dt\\
&=\widehat Q_\alpha(s)=\varphi_{K,A}^\alpha(s).         \end{align*}
By virtue of \eqref{H_to_alpha}, this yields  $\mathcal{H}_{K,A}^\alpha=\mathcal{H}_{K_\alpha,A}$, as desired.
\end{proof}

{\bf Remark 1}. As mentioned in \cite[Corollary 4]{faa}, under the assumptions of Theorem \ref{thm3} the operator $\mathcal{H}_{K,A}$
in $L^2(\mathbb{ R}^n_+)$ is unitary equivalent to the operator of multiplication by $ \varphi_{K,A}$ in $L^2(\mathbb{R}^n_+)$. It follows that
Theorem \ref{thm3} is valid for the space $L^2(\mathbb{R}^n_+)$ as well.

\medskip

{\bf Example 2}. Consider   the Calder\'on operator
$$
(\mathcal{K}f)(x)=\frac{1}{x} \int_0^xf(u)du+\int_x^{\infty}\frac{f(u)}{u}du
$$
in the space $L^2(\mathbb{R}_+)$.
This is a Hausdorff operator with
$$
K(u)=\frac{1}{u\max(1,u)}\chi_{(0,\infty)}(u),\quad A(u)=\frac{1}{u}.
$$

It follows that the  symbol of $\mathcal{K}$ is $\displaystyle{\varphi(s)=\frac1{s^2+\frac14}}$. Therefore, $\sigma(\mathcal{K})=[0,4]$. Further, let  $\rm{Re}\,\alpha>0 $.
It is known (see, e.~g., \cite[Ch. 1, \S 1.12 (40)]{BE}) that  $\displaystyle{\varphi^\alpha(s)=\frac1{(s^2+\frac14)^\alpha}}$ is the Fourier transform of the function
$$
Q_\alpha(t)=\pi^{-\frac12}\Gamma(\alpha)^{-1}|t|^{\alpha-\frac12}{\bf K}_{\alpha-\frac12}\left(\frac{|t|}{2}\right),
$$
where ${\bf K}_\nu$ is the function of  Macdonald.  Thus,
$$
(\mathcal{K}^\alpha f)(x)=\int_{0}^\infty K_\alpha (u)f\left(\frac{x}{u}\right)du,
$$
where the kernel $K_\alpha$ is given by formula \eqref{K_alpha}, with $\det A(u)=a(u)=\frac{1}{u}$,  $J(t)=-e^{-t}$, and the $Q_\alpha$ mentioned above.
In other words, $K_\alpha(u)=u^{-\frac32}Q_\alpha(\log u)$ for $u>0$. In particular,
$$
K_{\frac12+\imath\tau}(u)=\frac{1}{\sqrt{\pi}\Gamma(\frac12+\imath\tau) }u^{-\frac32}|\log u|^{\imath\tau}{\bf K}_{\imath\tau}\left(\frac{|\log u|}{2}\right)
$$
and
$$
(\mathcal{K}^{\frac12+\imath\tau} f)(x)=\frac{1}{\sqrt{\pi}\Gamma(\frac12+\imath\tau) }\int_{0}^\infty u^{-\frac32}|\log u|^{\imath\tau}{\bf K}_{\imath\tau}\left(\frac{|\log u|}{2}\right) f\left(\frac{x}{u}\right)du.
$$
Putting here $v=\frac1u$ and  $x=1$, we arrive at the following  index transform (for this class of integral transforms, see, e.~g., \cite{Yak})
\begin{eqnarray*}
 f^*(\tau):=\frac{1}{\sqrt{\pi}\Gamma(\frac12+\imath\tau) }\int_{0}^\infty v^{-\frac12}|\log v|^{\imath\tau}{\bf K}_{\imath\tau}\left(\frac{|\log v|}{2}\right) f\left(v\right)dv.
\end{eqnarray*}

\medskip







\section*{Competing Interests}

The authors have no relevant financial or non-financial interests to disclose.

\medskip

\section*{Data availability}

This manuscript has no associated data.

\end{document}